\documentclass[11 pt, draft]{article}
\title{On the self-intersections of curves deep in the lower central series of a surface group}
\author{Justin Malestein, Andrew Putman}
\date{November 30, 2009}
\usepackage{amsmath}
\usepackage{amssymb}
\usepackage{amsthm}
\usepackage{epsfig}
\usepackage[margin=1.25in]{geometry}
\usepackage{amsfonts}
\usepackage[font=small,format=plain,labelfont=bf,up,textfont=it,up]{caption}
\usepackage{amscd}
\usepackage{pinlabel}
\usepackage{stmaryrd}
\usepackage{type1cm}
\usepackage{mathptmx}  %roman=Times, math=Times where possible
\usepackage{calc}

\theoremstyle{plain}
\newtheorem{theorem}{Theorem}[section]

\newtheorem{lemma}[theorem]{Lemma}
\newtheorem{corollary}[theorem]{Corollary}

\newcommand\BeginCases{\setcounter{case}{0}}

\theoremstyle{definition}
\newtheorem{definition}[theorem]{Definition}

\newtheorem{case}{Case}

\theoremstyle{remark}
\newtheorem*{remark}{Remark}

\newtheorem*{acknowledgements}{Acknowledgments}

\DeclareMathOperator{\Ker}{ker}

\newcommand\R{\text{$\mathbb{R}$}}

\newcommand\Z{\text{$\mathbb{Z}$}}

\newcommand\N{\text{$\mathbb{N}$}}

\DeclareMathOperator{\HH}{H}

\DeclareMathOperator{\Min}{min}
\newcommand\Th{\text{th}}

\DeclareMathOperator{\Interior}{Int}

\newcommand\CaptionSpace{\hspace{0.2in}}

\newcommand\Figure[3]{
\begin{figure}[t]
\centering
\centerline{\psfig{file=#2,scale=60}}
\caption{#3}
\label{#1}
\end{figure}}

\newcommand\LCS[2]{\text{$\gamma_{#1}(#2)$}}
\newcommand\DER[2]{\text{$#2^{(#1)}$}}
\newcommand\DERR[2]{\text{$(#2)^{(#1)}$}}
\newcommand\LCSNorm[2]{\text{$m_{\text{lcs}}(#2,#1)$}}
\newcommand\DERNorm[2]{\text{$m_{\text{der}}(#2,#1)$}}
\newcommand\Length[2]{\text{$\| #2 \|_{#1}$}}

\begin{document}

\maketitle

\begin{abstract}
We give various estimates of the minimal number of self-intersections of a nontrivial
element of the $k^{\text{th}}$ term of the lower central series and derived series of the fundamental
group of a surface.  As an application, we obtain a new topological proof of the fact that
free groups and fundamental groups of closed surfaces are residually nilpotent.  Along the
way, we prove that a nontrivial element of the $k^{\text{th}}$ term of the lower central series
of a nonabelian free group has to have word length at least $k$ in a free generating set.
\end{abstract}

\section{Introduction}

Fix an orientable surface $\Sigma$.  The goal of this paper is to quantify the extent to
which algebraically complicated elements of $\pi_1(\Sigma)$ must exhibit topological complexity.  

We begin with some definitions.  Let $c : S^1 \rightarrow \Sigma$ be a closed curve.  We define
the {\em self-intersection number} of $c$, denoted $i(c)$, to be minimum over all curves $c'$ which are freely homotopic
to $c$ of the quantity 
$$\frac{1}{2}|\{\text{$(x,y)$ $|$ $x,y \in S^1$, $x \neq y$, $c'(x)=c'(y)$}\}|.$$
The factor $1/2$ appears because each self-intersection is counted twice.  Also, recall that if $G$ is a 
group, then the {\em lower central series} of $G$ is the inductively defined
sequence
$$\LCS{1}{G} = G \quad \text{and} \quad \LCS{k+1}{G} = [\LCS{k}{G},G].$$
For examples of curves in $\LCS{j}{\pi_1(\Sigma)}$ together with their self-intersection
numbers, see Figure \ref{figure:examples}.

\Figure{figure:examples}{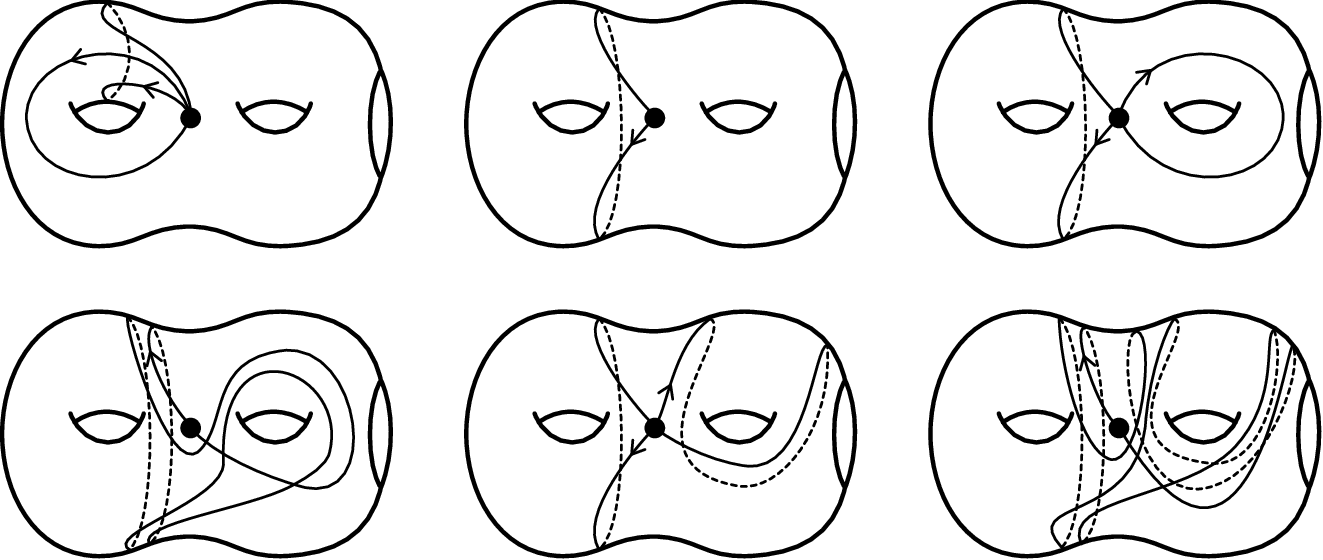}{
a,b. $z = [x,y] \in \LCS{2}{\pi_1(\Sigma)}$.  Also, $i(z)=0$. \CaptionSpace
c,d. $z = [x,y] \in \LCS{3}{\pi_1(\Sigma)}$ since $x \in \LCS{2}{\pi_1(\Sigma)}$.  Also, $i(z)=3$. \CaptionSpace
e,f. $z = [x,y] \in \LCS{4}{\pi_1(\Sigma)}$ since $x,y \in \LCS{2}{\pi_1(\Sigma)}$ (see Lemma \ref{lemma:threesubgroups}).
Also, $i(z)=3$.} 

If $\pi_1(\Sigma)$ is nonabelian, then it is easy
to see that for $k \geq 1$, there exist $x \in \LCS{k}{\pi_1(\Sigma)}$ with $i(x)$ arbitrarily large.  
However, a consequence of Theorems \ref{theorem:lcsbdry} and \ref{theorem:lcsgeneral} below is that
there do not exist nontrivial $x \in \LCS{k}{\pi_1(\Sigma)}$ with $i(x)$ arbitrarily small.

To state these theorems, we define
$$\LCSNorm{k}{\Sigma} = \Min\{\text{$i(x)$ $|$ $x \in \LCS{k}{\pi_1(\Sigma)}$, $x \neq 1$}\}.$$
Our first result is the following.
\begin{theorem}
\label{theorem:lcsbdry} 
Let $\Sigma_{g,b}$ be a orientable genus $g$ surface with $b \geq 1$ boundary components.  Assume that
$\pi_1(\Sigma_{g,b})$ is nonabelian.  Then for all $k \geq 1$ we have
$$\LCSNorm{k}{\Sigma_{g,b}} \geq \frac{k}{4g+b-1} - 1.$$
\end{theorem}
\noindent 
Theorem \ref{theorem:lcsbdry} will be proven in \S \ref{section:lcsbdry}.

The key to our proof of Theorem \ref{theorem:lcsbdry} is the following result, which is proven in \S
\ref{section:fox}.  If $G$ is a group and $S \subset G$, then
for $x \in \langle S \rangle$ we will denote by $\Length{S}{x}$ the length of the 
shortest word in $S \cup S^{-1}$ which equals $x$.
\begin{theorem}
\label{theorem:fox} 
Let $F(S)$ be the free group on a set $S$ with $|S| > 1$ and let $k \geq 1$.  Then for all
non-trivial $w \in \LCS{k}{F(S)}$ we have $k \leq \Length{S}{w}$.
\end{theorem}
\noindent
This improves upon work of Fox, who in \cite[Lemma 4.2]{FoxFree} proved a result
that implies that $\Length{S}{w} \geq \frac{1}{2} k$.

\begin{remark}
If we could prove an analogue of Theorem \ref{theorem:fox} for fundamental groups of closed surfaces, then we could also
prove an analogue of Theorem \ref{theorem:lcsbdry} for closed surfaces.
\end{remark}

\begin{remark}
We conjecture that Theorem \ref{theorem:fox} is not sharp.  Indeed,
we suspect that the length of the shortest word in the $k^{\Th}$ term of the lower central series
of a nonabelian free group is quadratic in $k$.  As evidence, in the proofs of
the upper bounds of Theorems \ref{theorem:lcsgeneral} and \ref{theorem:dergeneral} below
we will construct elements lying in the $k^{\Th}$ term of the lower central series
of a rank $2$ free group whose word length is quadratic in $k$.  If this conjecture
were true, then we could replace the lower bound in Theorem \ref{theorem:lcsbdry} with
a function which is quadratic in $k$.
\end{remark}

For general surfaces (not necessarily compact or of finite type), we prove the following.
\begin{theorem}
\label{theorem:lcsgeneral} 
Let $\Sigma$ be an orientable surface with $\pi_1(\Sigma)$ nonabelian.  Then for $k \geq 1$ we have
$$\log_8 (k) - 1 \leq \LCSNorm{k}{\Sigma} \leq 8 k^4.$$
\end{theorem}
\noindent
The proof of the lower bound in Theorem \ref{theorem:lcsgeneral} is
in \S \ref{section:lcsgenerallower} and the proof of the upper bound
is in \S \ref{section:upperbounds}.

\begin{remark}
Although the lower bound in Theorem \ref{theorem:lcsgeneral} is weaker than the lower bound in Theorem
\ref{theorem:lcsbdry} in terms of the order of $k$, it is uniform over all surfaces.
\end{remark}

Recall that a group $G$ is {\em residually nilpotent} if $\cap_{k=1}^{\infty} \LCS{k}{G} = 1$.  Our proof of Theorem
\ref{theorem:fox} is an elaboration of a proof due to Fox \cite{FoxFree} of a theorem of Magnus \cite{MagnusFree}
that says that free groups are residually nilpotent.  Conversely, an immediate consequence of Theorem
\ref{theorem:lcsgeneral} (which does not use Theorem \ref{theorem:fox}) is the following theorem, which for surface
groups is due independently to Baumslag \cite{BaumslagSurfaces} and Frederick \cite{FrederickSurfaces}.
\begin{corollary}
\label{corollary:residuallynilpotent}
Free groups and fundamental groups of closed surfaces are both residually nilpotent.
\end{corollary}
\noindent
Our proof of Theorem \ref{theorem:lcsgeneral} (and hence of Corollary \ref{corollary:residuallynilpotent})
shares some ideas with Hempel's beautiful short proof \cite{Hempel} of the residual finiteness of free
groups and surface groups.

The final result of this paper gives an analogue of Theorem \ref{theorem:lcsgeneral} for the derived
series.  Recall that if $G$ is a group, then the {\em derived series} of $G$ is
the inductively defined sequence
$$\DER{1}{G} = G \quad \text{and} \quad \DER{k+1}{G} = [\DER{k}{G},\DER{k}{G}].$$
Setting
$$\DERNorm{k}{\Sigma} = \Min\{\text{$i(x)$ $|$ $x \in \DERR{k}{\pi_1(\Sigma)}$, $x \neq 1$}\},$$
our result is as follows.
\begin{theorem}
\label{theorem:dergeneral} 
Let $\Sigma$ be an orientable surface with $\pi_1(\Sigma)$ nonabelian.  Then for $k \geq 3$ we have
$$2^{\lceil k/2 \rceil - 2} \leq \DERNorm{k}{\Sigma} \leq 2^{4k-5}.$$
\end{theorem}
\noindent
The lower bound in Theorem \ref{theorem:dergeneral} is proven in \S \ref{section:dergenerallower}
and the upper bound is proven in \S \ref{section:upperbounds}.  Our proof of the lower
bound in Theorem \ref{theorem:dergeneral} is inspired by an unpublished note of
Reznikov \cite{Reznikov}, which outlines an argument giving a linear
lower bound on $\DERNorm{k}{\Sigma}$ for $\Sigma$ closed.  We remark that though \cite{Reznikov}
seems to claim that it is dealing with the lower central series, both its definitions
and its arguments make it clear that the author intends to discuss the derived series.

\begin{remark}
In our definitions above, for $x \in \pi_1(\Sigma,\ast)$ the number $i(x)$ depends only on the free homotopy class of
$x$.  If we required that our homotopies fix $\ast$ and $\ast \in \Interior(\Sigma)$, then $i(x)$ would be unchanged.
If instead $\ast \in \partial \Sigma$, then $i(x)$ might differ. However, since the lower central series and derived
series are normal, requiring the homotopies to fix the basepoint would not change $\LCSNorm{k}{\Sigma}$ or
$\DERNorm{k}{\Sigma}$.
\end{remark}

\begin{acknowledgements}
We would like to thank
Khalid Bou-rabee, Nathan Broaddus, Matthew Day, Thomas Koberda, and Ben McReynolds 
for useful conversations and suggestions.  We would especially like to thank Benson
Farb for sharing \cite{Reznikov} with us and asking whether bounds of the sort we
prove might hold. 
\end{acknowledgements}

\section{Lower bounds}
\label{section:lowerbounds}

In this section, we prove the lower bounds in Theorems \ref{theorem:lcsbdry}, \ref{theorem:lcsgeneral},
and \ref{theorem:dergeneral}.

\subsection{Lower central series, compact surfaces with boundary}
\label{section:lcsbdry}

We begin with Theorem \ref{theorem:lcsbdry}.

\Figure{figure:lcspics}{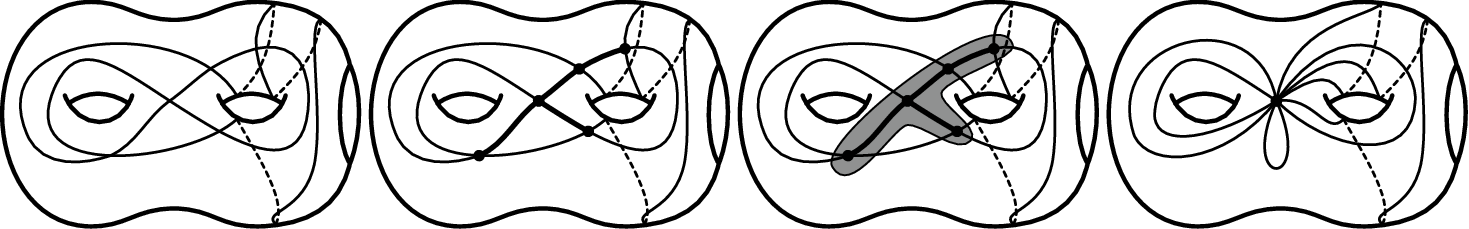}{a. An immersed curve $f$ whose singularities consist of $i(f)=5$ isolated
double points.
\CaptionSpace b. The maximal tree $T$
\CaptionSpace c. The $2$-disc $D$
\CaptionSpace d. Result of contracting $D$}

\begin{proof}[{Proof of Theorem \ref{theorem:lcsbdry}}]
Let $f : S^1 \rightarrow \Interior(\Sigma_{g,b})$ be an immersion whose singularities consist of $i(f)$ isolated double
points (see Figure \ref{figure:lcspics}.a).  
Assume that $f$ is freely homotopic to a nontrivial element of $\LCS{k}{\pi_1(\Sigma_{g,b})}$.
Our goal is to show that $i(f) \geq \frac{k}{4g+b-1} - 1$.

The first step is to ``comb'' the double points to a single point on the surface.  The immersion $f$
factors through an embedding of a graph whose vertices correspond to the singularities of $f$.  More precisely,
there is a $4$-regular graph $G$ with $i(f)$ vertices, an embedding $\tilde{f} : G \rightarrow \Interior(\Sigma_{g,b})$, and
an immersion $c : S^1 \rightarrow G$ with $f = \tilde{f} \circ c$ such that the inverse image under $c$ of
the interior of every edge of $G$ is connected.  

Let $T$ be a maximal 
tree in $G$.  Hence $\tilde{f}(T)$ is an embedded tree in $\Interior(\Sigma_{g,b})$ (see
Figure \ref{figure:lcspics}.b).  Any sufficiently small closed neighborhood $D$
of $\tilde{f}(T)$ satisfies the following two properties (see Figure \ref{figure:lcspics}.c).
\begin{itemize}
\item $D$ is homeomorphic to a closed $2$-disc.
\item For all edges $e$ of $G$ that do not lie in $T$, 
the set $\tilde{f}(e) \cap D$ has exactly two connected components.
\end{itemize}
It is easy to see that there is a map
$r : \Sigma_{g,b} \rightarrow \Sigma_{g,b}$ such that $r$ is homotopic
to the identity, such that $r|_{\Sigma_{g,b} \setminus D}$ is injective, and such
that $r(D) = \ast$ for some point $\ast \in \Interior(\Sigma_{g,b})$.  Let $D' = \tilde{f}^{-1}(D)$.  By
construction, $D'$ is a closed regular neighborhood of $T$ in $G$.
Set $G' = G / D'$, so $G'$ is a wedge of circles, and let $c' : S^1 \rightarrow G'$ 
be the composition
of $c$ with the projection $G \rightarrow G/D'$.  There is then an embedding 
$\tilde{f}' : G' \rightarrow \Interior(\Sigma_{g,b})$
such that $\tilde{f}' \circ c' = r \circ f$ (see Figure \ref{figure:lcspics}.d).

Let $w \in \pi_1(\Sigma_{g,b},\ast)$ be the based curve corresponding to $\tilde{f}' \circ c'$.  Since $\tilde{f}' \circ c'$
is freely homotopic to $f$, we have $w \in \LCS{k}{\pi_1(\Sigma_{g,b},\ast)}$.  Let
$S \subset \pi_1(\Sigma_{g,b},\ast)$ be a maximal collection of elements satisfying the following three properties.
\begin{itemize}
\item For each circle $L$ in $G'$ with $\tilde{f}'|_{L}$ not null-homotopic, there 
exists some $x \in S$ such that $\tilde{f}'|_{L} = x^{\pm 1}$.
\item For $x,y \in S$, if $x = y^{\pm 1}$ then $x=y$.
\item The curves in $S$ can be realized simultaneously by simple closed curves that only intersect at $\ast$.
\end{itemize}
Since $G$ is a $4$-regular graph with $i(f)$ vertices, it has $2 i(f)$ edges.  
Also, the maximal tree $T$ has $i(f)$ vertices and hence $i(f)-1$ edges.  We conclude that
$G'$ is a wedge of $2 i(f) - (i(f)-1) = i(f)+1$ circles, so $\Length{S}{w} \leq i(f)+1$.

We will confuse the set of homotopy classes $S$ with the corresponding set of simple closed curves that only
intersect at $\ast$.  Via an Euler characteristic calculation, we see that cutting $\Sigma_{g,b}$ along the curves in
$S$ yields $b$ annuli and $4g+b-2$ triangles.  By gluing the triangles together in an appropriate manner (as
in the standard combinatorial proof of the classification of surfaces; see \cite[Chapter 1]{Massey}), we identify
$\Sigma_{g,b}$ with a $(4g+b)$-sided polygon $P$ with $4g$ sides identified in pairs, all vertices identified, and
annuli glued to the $b$ unpaired sides.  Each of the curves in $S$ is identified with either a side of $P$ or an arc
in $P$ joining two vertices.

In particular, $S$ contains a free generating set $S'$ for $\pi_1(\Sigma_{g,b},\ast)$ consisting of the following curves.
\begin{itemize}
\item A curve corresponding to one edge from each of the pairs in the $4g$ paired edges in $P$.
\item A curve corresponding to all but one of the $b$ unpaired edges in $P$.
\end{itemize}
Observe that every element of $S$ can be written as a word of length at most $4g+b-1$ in $S'$, so $\Length{S'}{w} \leq
(4g+b-1) \Length{S}{w}$.  Theorem \ref{theorem:fox} says that $k \leq \Length{S'}{w}$, so we conclude that
$$k \leq \Length{S'}{w} \leq (4g+b-1) \Length{S}{w} \leq (4g+b-1) (i(f)+1).$$
Rearranging this inequality gives the desired conclusion.
\end{proof}

\subsection{Some preliminary lemmas}
We now prove two lemmas that are needed in the proofs of Theorems \ref{theorem:lcsgeneral} and 
\ref{theorem:dergeneral}.

\begin{lemma}
\label{lemma:coveringlemma} 
Let $\Sigma$ be a compact orientable surface with $\pi_1(\Sigma)$ non-abelian and let 
$f : S^1 \rightarrow \Sigma$ be a non-nullhomotopic closed curve.  Then there exists 
a degree $8$ normal cover $\widetilde{\Sigma} \rightarrow \Sigma$ such that one of the following holds.
\begin{itemize}
\item $f$ does not lift to a closed curve on $\widetilde{\Sigma}$.
\item $f$ lifts to a closed curve $\tilde{f} : S^1 \rightarrow \widetilde{\Sigma}$ with
$i(\tilde{f}) < i(f)$.
\end{itemize}
\end{lemma}

\begin{remark}
Since the cover in the conclusion of Lemma \ref{lemma:coveringlemma} is normal, $f$ lifts to a closed curve if and
only if any curve freely homotopic to $f$ lifts to a closed curve.
\end{remark}

\begin{proof}[{Proof of Lemma \ref{lemma:coveringlemma}}]
By the remark following the lemma, we may assume without loss of generality 
that $f$ is an immersion whose singularities consist of $i(f)$ isolated double points.  There are two cases.

\BeginCases
\begin{case}
$f$ is simple.
\end{case}
We must construct a degree $8$ normal cover to which $f$ does not lift to a closed curve.  In other words, 
choosing $\ast \in f(S^1)$ and letting $x \in \pi_1(\Sigma,\ast)$ be the based curve corresponding to $f$,
we must find a finite group $H$ with $|H|=8$ and a surjection $\psi : \pi_1(\Sigma,\ast) \rightarrow H$ 
with $x \notin \Ker(\psi)$.

If $f$ is not nullhomologous and if $\phi : \pi_1(\Sigma,\ast) \rightarrow \HH_1(\Sigma;\Z)$ is the
abelianization map, then $\phi(x)$ is a primitive vector.  There is therefore a surjection 
$\phi' : \HH_1(\Sigma;\Z) \rightarrow \Z / 8\Z$ such that $\phi'(\phi(x)) \neq 0$.  We conclude that 
we can use $H = \Z / 8 \Z$ and $\psi = \phi' \circ \phi$.

Assume now that $f$ is nullhomologous.  Letting $g$ be the genus and $b$ the number of boundary
components of $\Sigma$, it follows that 
there is a generating set $S=\{\alpha_1,\beta_1,\ldots,\alpha_g,\beta_g,x_1,\ldots,x_b\}$
for $\pi_1(\Sigma,\ast)$ such that
$$\pi_1(\Sigma,\ast) = \langle \text{$\alpha_1,\beta_1,\ldots,\alpha_g,\beta_g,x_1,\ldots,x_b$ $|$ $[\alpha_1,\beta_1] \cdots [\alpha_g,\beta_g] = x_1 \cdots x_b$} \rangle$$
and such that $x = [\alpha_1,\beta_1] \cdots [\alpha_{g'},\beta_{g'}]$ for some $g' \leq g$.  Let $H$ be the
dihedral group of order $8$, so
$$H = \langle \text{$\sigma,r$ $|$ $\sigma^2=1$, $r^4=1$, $\sigma r \sigma = r^{-1}$} \rangle.$$
We define a surjection $\psi : \pi_1(\Sigma,\ast) \rightarrow H$ in the following way.  If $b = 0$,
then $g' < g$ and we define $\psi(\alpha_1) = \psi(\alpha_g) = \sigma$,
$\psi(\beta_1) = \psi(\beta_g) = r \sigma$, and $\psi(s) = 1$ for all $s \in S$ with
$x \notin \{\alpha_1,\beta_1,\alpha_g,\beta_g\}$.
It is easy to check that the surface group relation is satisfied and that the resulting homomorphism $\psi$
is a surjection.  If $b > 0$,
then $\pi_1(\Sigma,\ast)$ is free on $S \setminus \{x_b\}$.  We define $\psi(\alpha_1) = \sigma$,
$\psi(\beta_1) = r \sigma$, and $\psi(s) = 1$ for all $s \in S \setminus \{x_b\}$ with
$s \notin \{\alpha_1,\beta_1,\alpha_g,\beta_g\}$.  Trivially
$\psi$ extends to a surjection.  In either case, we have $\psi(x) = [\sigma, r \sigma] \neq 1$, as desired.

\begin{case}
$f$ is not simple.
\end{case}

\Figure{figure:lcspics2}{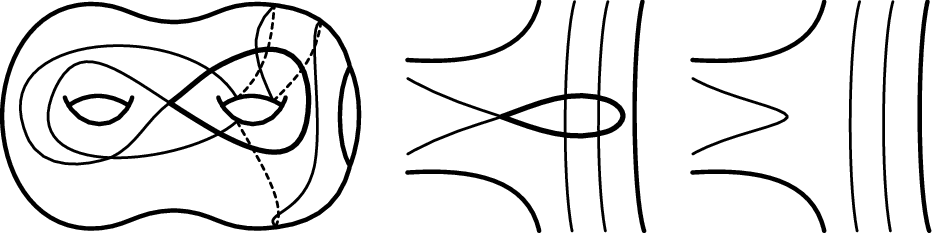}{
\CaptionSpace a. A nonsimple closed curve $f$ like in Step 2 of the proof of Lemma \ref{lemma:coveringlemma}.  The
simple closed subcurve $f'$ is in bold.
\CaptionSpace b. An example of a subcurve $f'$ that is nullhomotopic.
\CaptionSpace c. We reduce the number of self-intersections of $f$.}

Let $A$ be the set of nontrivial proper subarcs of $S^1$ whose endpoints are mapped by $f$ to the same point of
$\Sigma$.  By assumption $A$ is finite and nonempty.  Partially order the elements of $A$ by inclusion and let
$\alpha$ be a minimal element with endpoints $a_1$ and $a_2$.  Since $\alpha \in A$, the map $f|_{\alpha} : \alpha
\rightarrow \Sigma$ factors through a map $f' : S^1 \rightarrow \Sigma$, and from the minimality of
$\alpha$ we deduce that $f'$ is a simple closed curve (see Figure \ref{figure:lcspics2}.a). In addition, $f'$ is not
nullhomotopic, since if $f'$ were nullhomotopic then we could homotope $f$ so as to decrease its number of
self-intersections (see Figures \ref{figure:lcspics2}.b--c).

By Case 1, there is a degree $8$ normal cover $\widetilde{\Sigma} \rightarrow \Sigma$ to which $f'$ does not
lift to a closed curve.  If $f$ does not lift to a closed curve on $\widetilde{\Sigma}$, then we are done.  Assume,
therefore, that $f$ can be lifted to a closed curve $\tilde{f} : S^1 \rightarrow \widetilde{\Sigma}$.  Define
\begin{align*}
D(f) = \{\text{$(x,y)$ $|$ $x,y \in S^1$, $x \neq y$, $f(x)=f(y)$}\},\\
D(\tilde{f}) = \{\text{$(x,y)$ $|$ $x,y \in S^1$, $x \neq y$, $\tilde{f}(x)=\tilde{f}(y)$}\}.
\end{align*}
We clearly have $D(\tilde{f}) \subset D(f)$.  Moreover, by construction $(a_1,a_2) \notin D(\tilde{f})$.  We conclude
that $\tilde{f}$ has fewer self-intersections than $f$, so $i(\tilde{f}) < i(f)$, as desired.
\end{proof}

We will also need the following simple lemma, which allows us to deduce results about noncompact surfaces
from results about compact surfaces.

\begin{lemma}
\label{lemma:reducetocompact}
Let $\Sigma$ be an oriented surface with $\pi_1(\Sigma)$ nonabelian.  Also, let $f: S^1 \rightarrow \Sigma$ 
be a non-nullhomotopic closed curve which is freely homotopic to an element of $\LCS{k}{\pi_1(\Sigma)}$ for some
$k \geq 1$.  Then there is a compact surface $\Sigma'$ with $\pi_1(\Sigma')$ nonabelian and an
embedding $i : \Sigma' \hookrightarrow \Sigma$ satisfying the following properties.
\begin{itemize}
\item There is a map $f' : S^1 \rightarrow \Sigma'$ such that $f = i \circ f$.
\item The curve $f'$ is freely homotopic to an element of $\LCS{k}{\pi_1(\Sigma')}$.
\end{itemize}
\end{lemma}
\begin{proof}
Any iterated commutator only involves a finite number of curves and any homotopy stays within a compact subset of
$\Sigma$.
\end{proof}

\subsection{Lower central series, general surfaces}
\label{section:lcsgenerallower}

We now prove the lower bound in Theorem \ref{theorem:lcsgeneral}.  The proof will require the following 
lemma.

\begin{lemma}
\label{lemma:grouptheory} Fix $p,n,m \geq 1$ with $p$ prime, and let $G_0 \rhd G_1 \rhd \cdots \rhd G_n$ be a
subnormal sequence of groups with $[G_{i-1}:G_i] = p^{m}$ for $1 \leq i \leq n$.  Then there exists some
group $H$ such that $H < G_n$, such that $H \lhd G_0$, and such that $[G_0 : H] = p^{N}$ for some $1 \leq N \leq m
\frac{p^{mn} - 1}{p^m - 1}$.
\end{lemma}

\noindent For the proof of Lemma \ref{lemma:grouptheory}, we will need the following.

\begin{lemma}
\label{lemma:grouptheory2}
Fix $p,r,s \geq 1$ with $p$ prime, and let $A \rhd B \rhd C$ be groups with
$[A:B] = p^{r}$ and $[B:C] = p^{s}$.  Then there exists a group $D$ such that $D < C$, such that
$D \lhd A$, and such that $[A:D] = p^N$ for some $1 \leq N \leq p^r s + r$.
\end{lemma}
\begin{proof}
Define $D = \bigcap_{a \in A} a^{-1} C a$.  Clearly we have $D < C$ and $D \lhd A$, so we must only
prove the indicated result about $[A:D]$.  Let $T = \{a_1,\ldots,a_{p^{r}}\}$ be a complete set of coset
representatives for $B$ in $A$ with $a_1 = 1$.  Hence we have $D = \bigcap_{j=1}^{p^r} a_j^{-1} C a_j$.  For
$1 \leq i \leq p^{r}$, define $C_i = \bigcap_{j=1}^{i} a_j^{-1} C a_j$.  We thus have
$$A \rhd B \rhd C = C_1 \rhd C_2 \rhd \cdots \rhd C_{p^{r}} = D.$$
We claim that for $1 < i \leq p^{r}$ we have $[C_{i-1}:C_{i}] = p^{k_i}$ for some $0 \leq k_i \leq s$.  Indeed, we have
$$C_{i-1} / C_{i} = C_{i-1} / (a_{i}^{-1} C a_i \cap C_{i-1}) \cong (C_{i-1} \cdot (a_i^{-1} C a_i)) / a_i^{-1} C a_i < B / a_i^{-1} C a_i.$$
Since $[B:a_i^{-1} C a_i] = [B:C] = p^s$, the claim follows.  We conclude that
\begin{align*}
[A:D] &= [A:B][B:C][C_1:C_2] \cdots [C_{p^r-1}:C_{p^r}] = p^r p^s p^{k_2} \cdots p^{k_{p^r}} \leq p^r (p^s)^{p^r},
\end{align*}
as desired.
\end{proof}

\begin{proof}[{Proof of Lemma \ref{lemma:grouptheory}}]
The proof will be by induction on $n$.  The base case $n=1$ is trivial.  Now assume that $n > 1$ and that the lemma
is true for all smaller $n$.  Applying the inductive hypothesis to the sequence $G_1 \rhd \cdots \rhd G_n$, we obtain
a group $H'$ such that $H' < G_n$, such that $H' \lhd G_1$, and such that $[G_1:H'] = p^{N'}$ with 
$N' \leq m \frac{p^{m(n-1)} - 1}{p^m - 1}$.  We can therefore apply Lemma \ref{lemma:grouptheory2} to the sequence 
$G_0 \rhd G_1 \rhd H'$ and obtain a group $H$ such that $H < H' < G_n$, such that $H \lhd G_0$, and 
such that $[G_0 : H] = p^N$ for some $N$ that satisfies
$$N \leq p^{m} N' + m \leq p^{m} m \frac{p^{m(n-1)} - 1}{p^m - 1} + m = m \frac{p^{mn}- p^m}{p^m - 1} + m = m \frac{p^{mn} - 1}{p^m - 1},$$ 
as desired.
\end{proof}

We will also need the following standard property of $p$-groups.  Recall that a group $G$ is
{\it at most $n$-step nilpotent} if $\LCS{n}{G} = 1$.

\begin{lemma}[{\cite[Theorem 5.33]{RotmanBook}}]
\label{lemma:pgroups}
Let $p$ be a prime and let $G$ be a group with $|G|=p^n$ for some $n \in \N$.  Then
$G$ is at most $n$-step nilpotent.
\end{lemma}

We can now prove the lower bound in Theorem \ref{theorem:lcsgeneral}.

\begin{proof}[{Proof of Theorem \ref{theorem:lcsgeneral}, lower bound}]
Let $f : S^1 \rightarrow \Sigma$ be an immersion whose singularities consist of $i(f)$ isolated double
points.  Assume that $f$ is freely homotopic to a nontrivial element of $\LCS{k}{\pi_1(\Sigma)}$.
Our goal is to show that $i(f) \geq \log_8 (k) - 1$; i.e.\ that $k \leq 8^{i(f)+1}$.

By Lemma \ref{lemma:reducetocompact}, we may assume that $\Sigma$ is compact.  Choose a basepoint
$\ast \in f(S^1)$ and let $x \in \pi_1(\Sigma,\ast)$ be the based curve corresponding to $f$.
Applying Lemma \ref{lemma:coveringlemma} repeatedly, we obtain a subnormal sequence
$$\pi_1(\Sigma,\ast) = G_0 \rhd G_1 \rhd \cdots \rhd G_n$$
with $n \leq i(f) + 1$ such that $x \notin G_n$ and such that $[G_{i-1}:G_i] = 2^{3}$ 
for $1 \leq i \leq n$.  Applying Lemma \ref{lemma:grouptheory},
we obtain a group $H$ such that $H < G_n$, such that $H \lhd \pi_1(\Sigma,\ast)$, and such that
$[\pi_1(\Sigma,\ast):H] = 2^N$ for some 
$$N \leq 3 \frac{2^{3n} - 1}{2^3 - 1} \leq 8^n \leq 8^{i(f)+1}.$$
By Lemma \ref{lemma:pgroups}, we deduce that $\pi_1(\Sigma,\ast) / H$ is at most $8^{i(f)+1}$-step nilpotent.  In
other words, 
$$\LCS{8^{i(f)+1}}{\pi_1(\Sigma,\ast)} < H.$$  
Since $H$ is a normal subgroup of $\pi_1(\Sigma,\ast)$ and $f$ is freely homotopic to $x \notin H$, it
follows that $f$ is not freely homotopic to any element of $H$.  We conclude that $k \leq 8^{i(f)+1}$, as
desired. 
\end{proof}

\subsection{Derived series}
\label{section:dergenerallower}

We now prove the lower bound in Theorem \ref{theorem:dergeneral}.  The
proof will require the following lemma.

\begin{lemma}
\label{lemma:simpleclosedcurves}
Let $\Sigma$ be an orientable surface (not necessarily compact) with $\pi_1(\Sigma)$ nonabelian.
Also, let $f : S^1 \rightarrow \Sigma$ be a non-nullhomotopic simple closed
curve.  Then $f$ is not freely homotopic to any element of $\LCS{3}{\pi_1(\Sigma)}$.
\end{lemma}
\begin{proof}
By Lemma \ref{lemma:reducetocompact}, we may assume that $\Sigma$ is compact.
Assume that $f$ is freely homotopic to $x \in \pi_1(\Sigma)$.
Since $f$ is simple, Lemma \ref{lemma:coveringlemma} implies that there is a finite group $H$ with
$|H| = 2^3$ and a surjection $\psi : \pi_1(\Sigma) \rightarrow H$ such that $\psi(x) \neq 1$.
Lemma \ref{lemma:pgroups} says that $H$ is at most $3$-step nilpotent, so $\LCS{3}{\pi_1(\Sigma)} \subset \Ker(\psi)$.
We conclude that $x \notin \LCS{3}{\pi_1(\Sigma)}$, as desired.
\end{proof}

We will also need the following standard lemma.

\begin{lemma}[{\cite[Exercise 5.50]{RotmanBook}}]
\label{lemma:threesubgroups}
If $G$ is a group, then for all $k \geq 1$ we have $\DER{k}{G} < \LCS{2^{k-1}}{G}$.
\end{lemma}

We can now prove the lower bound in Theorem \ref{theorem:dergeneral}.

\begin{proof}[{Proof of Theorem \ref{theorem:dergeneral}, lower bound}]
We will prove that $2^{\lceil k/2 \rceil - 2} \leq \DERNorm{k}{\Sigma}$ for $k \geq 3$ by induction on $k$.  The base
cases $k=3$ and $k=4$ follow from Lemma \ref{lemma:simpleclosedcurves} combined with Lemma
\ref{lemma:threesubgroups}.  Now assume that $k > 4$ and that the result is true for all smaller $k$.  It is enough
to prove that 
$$\DERNorm{k}{\Sigma} \geq 2 \cdot \DERNorm{k-2}{\Sigma}.$$  
Consider an immersion $f : S^1 \rightarrow \Sigma$
whose singularities consist of $i(f)$ isolated double points.  Assume that $i(f) < 2 \cdot \DERNorm{k-2}{\Sigma}$.
Our goal is to show that $f$ is not freely homotopic to any element of $\DERR{k}{\pi_1(\Sigma)}$.

Let $\pi : \widetilde{\Sigma} \rightarrow \Sigma$ be the normal covering corresponding to the subgroup
$\DERR{k-2}{\pi_1(\Sigma)}$.  If $f$ does not lift to a closed curve in $\widetilde{\Sigma}$, then $f$ is not freely
homotopic to any element of $\DERR{k-2}{\pi_1(\Sigma)}$, and thus is certainly not freely
homotopic to any element of $\DERR{k}{\pi_1(\Sigma)}$.
Assume, therefore, that there is a lift $\tilde{f} : S^1 \rightarrow \widetilde{\Sigma}$ of $f$. We claim that
$\tilde{f}$ is a simple closed curve.  Indeed, define
\begin{align*}
D(f) = \{\text{$(x,y)$ $|$ $x,y \in S^1$, $x \neq y$, $f(x)=f(y)$}\},\\
D(\tilde{f}) = \{\text{$(x,y)$ $|$ $x,y \in S^1$, $x \neq y$, $\tilde{f}(x)=\tilde{f}(y)$}\}.
\end{align*}
Clearly $D(\tilde{f}) \subset D(f)$, and we want to prove that $D(\tilde{f}) = \emptyset$. Consider any $(x,y) \in
D(f)$.  The points $x$ and $y$ divide $S^1$ into two arcs $\alpha$ and $\alpha'$, and the restrictions of $f$ to both
$\alpha$ and $\alpha'$ are closed curves.  The number of self-intersections of one of $f|_{\alpha}$ and
$f|_{\alpha'}$ (say $f|_{\alpha}$) is less than half of the number of self-intersections of $f$.  Hence the
closed curve defined by $f|_{\alpha}$ has fewer than than $\DERNorm{k-2}{\Sigma}$ self-intersections, so it is not
freely homotopic to any element of $\DERR{k-2}{\pi_1(\Sigma)}$.  We conclude that $\tilde{f}|_{\alpha}$ is not a
closed curve, so $(x,y) \notin D(\tilde{f})$, as desired.

Observe now that by Lemmas \ref{lemma:simpleclosedcurves} and \ref{lemma:threesubgroups}, the curve
$\tilde{f}$ is not freely homotopic to any element of $\DERR{3}{\pi_1(\widetilde{\Sigma})}$.  Since
$$\DERR{3}{\pi_1(\widetilde{\Sigma})} = \DERR{3}{\DERR{k-2}{\pi_1(\Sigma)}} = \DERR{k}{\pi_1(\Sigma)},$$
we conclude that $f$ is not freely homotopic to any element of $\DERR{k}{\pi_1(\Sigma)}$, as desired.
\end{proof}

\section{Upper bounds}
\label{section:upperbounds}

We now prove the upper bounds in Theorems \ref{theorem:lcsgeneral} and \ref{theorem:dergeneral}.  We will need two
lemmas.

\begin{lemma}
\label{lemma:wordlength}
Let $(\Sigma,\ast)$ be a based surface and let $S \subset \pi_1(\Sigma,\ast)$ be a set
consisting of elements that can be realized simultaneously by simple closed curves that only intersect at $\ast$.  Then
for all $x \in \langle S \rangle \subset \pi_1(\Sigma,\ast)$, we have $i(x) \leq \binom{\Length{S}{x}}{2}$.
\end{lemma}
\begin{proof}
We can assume that $\ast \in \Interior(\Sigma)$. Set $n = \Length{S}{x}$ and write $x = s_1 \cdots s_n$ with $s_i \in
S \cup S^{-1}$ for $1 \leq i \leq n$.  For $1 \leq i \leq n$, we can choose embeddings $f_i : S^1 \rightarrow \Sigma$
such that $f_i$ represents $s_i$.  Moreover, we can choose the $f_i$ such that 
$f_i(S^1) \cap f_j(S^1) = \{\ast\}$ for $1 \leq i < j \leq n$.
Let $D \subset \Sigma$ be a closed embedded $2$-disc with $\ast \in D$ such that
$f_i(S^1) \cap D$ is a connected arc for
all $1 \leq i \leq n$.
Parametrize $D$ such that $D$ is the unit
disc in $\R^2$ and $\ast = (0,0)$.  For $1 \leq i \leq n$, let $f'_i : [0,1] \rightarrow \Sigma$ be a parametrization
of the oriented arc $f_i(S^1) \setminus \Interior(D)$. Observe that for $1 \leq i < j \leq n$ we have $f'_i([0,1])
\cap f'_j([0,1]) = \emptyset$.

We can now construct a curve $f : S^1 \rightarrow \Sigma$ that is freely homotopic to $x$ in the following
way.  The curve $f$ first traverses $f'_1$, then goes along a straight line in $D$ from $f'_1(1)$ to $f'_2(0)$, then
traverses $f'_2$, then goes along a straight line in $D$ from $f'_2(1)$ to $f'_3(0)$, then traverses $f'_3$, etc.  The
curve $f$ ends with a straight line in $D$ from $f'_n(1)$ to $f'_1(0)$.  Clearly $f$ is freely homotopic to $x$.  Moreover,
all self-intersections of $f$ must occur in $D$.  Since $f(S^1) \cap D$ consists of $n$ straight lines and any two of these
lines can intersect at most once, we conclude that $f$ has at most $\binom{n}{2}$ self-intersections, as desired.
\end{proof}

\begin{lemma}
\label{lemma:freegroup}
Let $S = \{a_1,a_2\}$ and let $F_S$ be the free group on $S$.  Then for all $k \geq 1$ there
exists some $w \in F_S$ with $w \neq 1$ such that $\Length{S}{w} \leq 4^{k-1}$ and $w \in \DER{k}{F_S}$.
\end{lemma}
\begin{proof}
Define elements $x_k$ and $y_k$ inductively as follows.
\begin{align*}
x_1 = a_1 \quad &\text{and} \quad y_1 = a_2,\\
x_k = [x_{k-1},y_{k-1}] \quad &\text{and} \quad y_k=[x_{k-1},y_{k-1}^{-1}].
\end{align*}
Clearly $\Length{S}{x_k} \leq 4^{k-1}$ and $x_k \in \DER{k}{F_S}$ for $k \geq 1$.  We
must therefore only prove that $x_k \neq 1$ for $k \geq 1$.  In fact, we will
prove by induction on $k$ that $x_k$ and $y_k$ generate a rank $2$ free subgroup of $F_S$ for $k \geq 1$.
The base case $k=1$ is trivial.  Now assume that $k > 1$ and that $x_{k-1}$ and $y_{k-1}$ generate
a rank $2$ free subgroup.  Since neither $x_k$ nor $y_k$ is trivial, they must generate either a rank
$2$ or rank $1$ free subgroup.  But since $x_{k-1}$ and $y_{k-1}$ generate a rank $2$ free subgroup, we have
$$[x_k,y_k] = [[x_{k-1},y_{k-1}],[x_{k-1},y_{k-1}^{-1}]] \neq 1,$$
so we conclude that $x_k$ and $y_k$ cannot generate a rank $1$ subgroup.
\end{proof}

We can now prove the upper bounds in Theorems \ref{theorem:lcsgeneral} and \ref{theorem:dergeneral}.

\begin{proof}[{Proof of Theorem \ref{theorem:dergeneral}, upper bound}]
We wish to prove that $\DERNorm{k}{\Sigma} \leq 2^{4k-5}$ for $k \geq 3$.  In fact, this inequality holds
for $k \geq 1$ (the assumption that $k \geq 3$ is necessary only in the lower bound), so fix
$k \geq 1$.  We claim that there
exists some $a_1,a_2 \in \pi_1(\Sigma,\ast)$ that generate a rank $2$ free subgroup of $\pi_1(\Sigma,\ast)$ and
can be realized simultaneously by simple closed curves that only intersect at $\ast$.  If $\Sigma$ is compact, then this
is trivial.  Otherwise, $\pi_1(\Sigma,\ast)$ must be a nonabelian free group (see, e.g., \cite[\S 44A]{AhlforsRiemann}), 
so we can find $a_1',a_2' \in \pi_1(\Sigma,\ast)$
that generate a rank $2$ free subgroup.  Like in the proof of the Theorem \ref{theorem:lcsbdry}, we can
``comb'' the intersections and self-intersections of $a_1'$ and $a_2'$ to $\ast$ and find a set $S' \subset \pi_1(\Sigma,\ast)$ 
of elements that can be realized simultaneously by simple
closed curves that only intersect at $\ast$ such that both $a_1'$ and $a_2'$ can be expressed as products
of elements of $S' \cup (S')^{-1}$.  There must then exist $a_1,a_2 \in S'$ that generate a rank $2$ free
subgroup, as desired.

Set $S = \{a_1,a_2\}$.  By Lemma \ref{lemma:freegroup},
there is some $w \in \langle S \rangle$ such that $\Length{S}{w} \leq 4^{k-1}$ and 
$w \in \DERR{k}{\pi_1(\Sigma)}$.  By Lemma \ref{lemma:wordlength}, we deduce that
$$i(w) \leq \binom{\Length{S}{x}}{2} \leq \frac{4^{k-1}(4^{k-1}-1)}{2} \leq \frac{1}{2} 4^{2k-2} = 2^{4k-5},$$
so $\DERNorm{k}{\Sigma} \leq 2^{4k-5}$, as desired.
\end{proof}

\begin{proof}[{Proof of Theorem \ref{theorem:lcsgeneral}, upper bound}]
Fix $k \geq 1$.  We can then find an integer $l$ such that $\log_2(k) \leq l-1 \leq \log_2(k)+1$.  The upper bound of
Theorem \ref{theorem:dergeneral} (which as we observed above holds for $k \geq 1$) implies that we can find 
$x \in \DERR{l}{\pi_1(\Sigma)}$ such that
$$i(x) \leq 2^{4l-5} \leq 2^{4 (\log_2(k)+2)-5} = 8k^4.$$
By Lemma \ref{lemma:threesubgroups}, we have $x \in \LCS{k}{\pi_1(\Sigma)}$, so we conclude that $\LCSNorm{k}{\Sigma}
\leq 8k^4$, as desired.
\end{proof}

\section{Word length in the lower central series}
\label{section:fox}

In this section, we will prove Theorem \ref{theorem:fox}.  As was indicated in the introduction,
this proof is inspired by an argument of Fox \cite[Lemma 4.2]{FoxFree}.  Our main tool
will be the Fox free differential calculus, so we begin by recalling a number of basic
facts about this calculus.  A good reference is \cite{FoxFree}.

Let $F$ be the free group on a set $S$ and let $\varepsilon : \Z F \rightarrow \Z$ be
the augmentation map; i.e.\ the unique linear map with $\varepsilon(g) = 1$ for all $g \in F(S)$.

\begin{definition}
A {\em free derivative} is a linear map $D : \Z F \rightarrow \Z F$ such that
$D(xy) = (D(x))\varepsilon(y) + x D(y)$ for all $x,y \in \Z F$.
\end{definition}

An easy induction establishes that if $D$ is a free derivative, then for $v_1,\ldots,v_k \in \Z F$
we have
\begin{equation}
\label{eqn:productrule}
D(v_1 \cdots v_k) = \sum_{i=1}^k (v_1 \cdots v_{i-1})(D(v_i))\varepsilon(v_{i+1}) \cdots \varepsilon(v_{k}).
\end{equation}
A consequence of \eqref{eqn:productrule} is that for $g \in F$, we have
\begin{equation}
\label{eqn:inverserule}
D(g^{-1}) = - g^{-1} D(g).
\end{equation}
The basic existence result for free derivatives is the following.

\begin{lemma}[{\cite[\S 2]{FoxFree}}]
\label{lemma:freederivexis}
For every $s \in S$, there is a unique free derivative $D_s$ satisfying $D_s(s) = 1$ and $D_s(s') = 0$
for $s' \in S$ with $s' \neq s$.  
\end{lemma}
\noindent
By \eqref{eqn:productrule} and \eqref{eqn:inverserule}, we have
\begin{equation}
\label{eqn:exprule}
\varepsilon(D_s(s^k)) = k
\end{equation}
for all $s \in S$ and $k \in \Z$.

For $k \geq 1$ and $s_1,\ldots,s_k \in S$, we will call 
the product $D_{s_1} \cdots D_{s_k}$ a {\em free derivative of order $k$}.  
The basic fact connecting the Fox free differential calculus to
the lower central series of $F$ is the following easy lemma.

\begin{lemma}[{\cite[3.1]{ChenFoxLyndon}}]
\label{lemma:freelcs}
For $k \geq 2$ and $g \in \LCS{k}{F}$, we have $\epsilon(D(g)) = 0$ for all free derivatives $D$
of order less than or equal to $k-1$.
\end{lemma}

We can now prove Theorem \ref{theorem:fox}.

\begin{proof}[{Proof of Theorem \ref{theorem:fox}}]
Consider $w \in \LCS{k}{F(S)}$ with $w \neq 1$.  Our goal is to show that $k \leq \Length{S}{w}$.  We will
produce a free derivative $D$ whose order is at most $\Length{S}{w}$ such that $\epsilon(D(w)) \neq 0$.
By Lemma \ref{lemma:freelcs}, it will follow that 
$$w \notin \LCS{1+\Length{S}{w}}{F(S)},$$
and hence that $k \leq \Length{S}{w}$.

Write $w = u_1 \cdots u_n$ with $u_i = s_i^{m_i}$ for some $s_i \in S$ and $m_i \in \Z \setminus \{0\}$ for $1 \leq i \leq n$.
Choose this expression such that $s_i \neq s_{i+1}$ for $1 \leq i < n$.  We thus have 
$n \leq \Length{S}{w}$.  Define $D = D_{s_1} \cdots D_{s_n}$.
We must show that $\varepsilon(D(w)) \neq 0$.  In fact, we will show that for all $1 \leq j \leq n$ we have
\begin{align}
&D_{s_{j}} D_{s_{j+1}} \cdots D_{s_{n}}(w) \label{eqn:biggoal}\\
&\quad\quad = \sum_{1 \leq i_{j} < i_{j+1} < \cdots < i_n \leq n} (u_1 \cdots u_{i_j-1})(D_{s_j}(u_{i_j})) \varepsilon(D_{s_{j+1}}(u_{i_{j+1}})) \cdots \varepsilon(D_{s_n}(u_{i_n})). \notag
\end{align}
In particular, the case $j=1$ will yield
$$D(w) = D_{s_1}(u_1) \varepsilon(D_{s_2}(u_2)) \cdots \varepsilon(D_{s_n}(u_n)).$$
Using \eqref{eqn:exprule}, we will then be able to deduce that
$$\varepsilon(D(w)) = \varepsilon(D_{s_1}(u_1)) \cdots \varepsilon(D_{s_n}(u_n)) = m_1 \cdots m_n \neq 0,$$
as desired.

The proof of \eqref{eqn:biggoal} will be by induction on $n-j$.  The base case $n-j=0$ follows from
\eqref{eqn:productrule} and the fact that $\varepsilon(u_i) = 1$ for all $1 \leq i \leq n$.  Now assume
that $n-j > 0$ and that \eqref{eqn:biggoal} holds for all smaller $n-j$.  Since $s_i \neq s_{i+1}$ for
$1 \leq i < n$, we must have $D_{s_j} D_{s_{j+1}}(u_{j+1}) = 0$.  Using this together with \eqref{eqn:productrule}, our inductive
hypothesis, and the fact that $\varepsilon(u_i) = 1$ for all $1 \leq i \leq n$, we obtain
\begin{align*}
&D_{s_{j}} D_{s_{j+1}} \cdots D_{s_{n}}(w) \\
&\quad\quad = D_{s_j}(\sum_{1 \leq i_{j+1} < \cdots < i_n \leq n} (u_1 \cdots u_{i_{j+1}-1})(D_{s_{j+1}}(u_{i_{j+1}})) \varepsilon(D_{s_{j+2}}(u_{i_{j+2}})) \cdots \varepsilon(D_{s_n}(u_{i_n})) \\
&\quad\quad = \sum_{1 \leq i_{j+1} < \cdots < i_n \leq n} (\sum_{i=1}^{i_{j+1}-1} (u_1 \cdots u_{i-1})(D_{s_j}(u_i)) \varepsilon(D_{s_{j+1}}(u_{i_{j+1}})) \cdots \varepsilon(D_{s_n}(u_{i_n}))) \\
&\quad\quad = \sum_{1 \leq i_{j} < i_{j+1} < \cdots < i_n \leq n} (u_1 \cdots u_{i_j-1})(D_{s_j}(u_{i_j})) \varepsilon(D_{s_{j+1}}(u_{i_{j+1}})) \cdots \varepsilon(D_{s_n}(u_{i_n})),
\end{align*}
and we are done.
\end{proof}

\noindent
\begin{tabular*}{\linewidth}[t]{@{}p{\widthof{E-mail: {\tt justinm@math.uchicago.edu}}+1in}@{}p{\linewidth - \widthof{E-mail: {\tt justinm@math.uchicago.edu}}-1in}@{}}
{\raggedright 
Justin Malestein\\
Department of Mathematics\\
University of Chicago\\
5734 University Avenue\\
Chicago, IL 60637-1514\\
E-mail: {\tt justinm@math.uchicago.edu}} 
& 
{\raggedright
Andrew Putman\\
Department of Mathematics\\
MIT, 2-306 \\
77 Massachusetts Avenue \\
Cambridge, MA 02139-4307 \\
E-mail: {\tt andyp@math.mit.edu}}
\end{tabular*}


\begin{thebibliography}{}
\begin{small}
\setlength{\itemsep}{1pt}

% verified
\bibitem{AhlforsRiemann}
L. V. Ahlfors\ and\ L. Sario, 
{\it Riemann surfaces}, 
Princeton Univ. Press, Princeton, N.J., 1960.

% verified
\bibitem{BaumslagSurfaces}
G. Baumslag,
On generalised free products,
Math. Z. {\bf 78} (1962), 423--438.

% verified 
\bibitem{ChenFoxLyndon}
K.-T. Chen, R. H. Fox\ and\ R. C. Lyndon, 
Free differential calculus. IV. The quotient groups of the lower central series, 
Ann. of Math. (2) {\bf 68} (1958), 81--95.

% verified
\bibitem{FoxFree}
R. H. Fox,
Free differential calculus. I. Derivation in the free group ring,
Ann. of Math. (2) {\bf 57} (1953), 547--560.

% verified
\bibitem{FrederickSurfaces}
K. N. Frederick,
The Hopfian property for a class of fundamental groups,
Comm. Pure Appl. Math. {\bf 16} (1963), 1--8.

% verified
\bibitem{Hempel}
J. Hempel, 
Residual finiteness of surface groups, 
Proc. Amer. Math. Soc. {\bf 32} (1972), 323.

% verified
\bibitem{MagnusFree}
W. Magnus,
Beziehungen zwischen Gruppen und Idealen in einem speziellen Ring,
Math. Ann. {\bf 111} (1935), no.~1, 259--280.

% verified
\bibitem{Massey}
W. S. Massey, 
{\it Algebraic topology: An introduction}, 
Harcourt, Brace \& World, Inc., New York, 1967.

% verified
\bibitem{Reznikov}
A. Reznikov,
Crossing number and lower central series of a surface group,
unpublished preprint, 1998.

% verified
\bibitem{RotmanBook}
J. J. Rotman,
{\it An introduction to the theory of groups},
Fourth edition, Springer, New York, 1995.

\end{small}
\end{thebibliography}
\end{document}